\documentclass{article}
\usepackage{latexsym}
\usepackage{mathrsfs}
\usepackage{indentfirst}
\usepackage{amsmath,amsthm, amssymb}
\usepackage{graphicx}
\usepackage{setspace}
\usepackage{setspace}
\usepackage{geometry}
\usepackage{float}
\usepackage{epsfig}
\usepackage{booktabs}
\usepackage{verbatim}
\usepackage{multirow}
\usepackage{tikz}
\usepackage{tikz-cd}
\usepackage{csquotes}
\usepackage{braids}
\usepackage[hidelinks]{hyperref}

\newtheorem{theorem}{Theorem}[section]
\newtheorem{proposition}[theorem]{Proposition}
\newtheorem{lemma}[theorem]{Lemma}
\newtheorem{corollary}[theorem]{Corollary}

\newtheorem{remark}[theorem]{Remark}

\begin{document}
\title{Second-to-Top Term of $\widehat{HFK}$ of Closed 3-Braids}
\author{Zhaojun Chen$^1$\thanks{Email address: zchen5@caltech.edu.}
\\ $^1$Department of Mathematics,
\\ California Institute of Technology, Pasadena, CA 91125, United States
}
\date{Oct 15, 2025}
\maketitle
%removing the first page number
%\thispagestyle{empty}

%abstract part
~\\
\begin{abstract}
In this paper, we use the skein exact sequence and other techniques to compute the second-to-top term of $\widehat{HFK}$ of closed 3-braids. We do it case-by-case according to Xu's classification. We also verify the rank inequality conjectured by Sivek.
%key words
$\newline$\noindent\textbf{Key words}: Closed 3-Braids, Knot Floer Homology, Skein Relationship, Quasi-Alternating Link.
\end{abstract}
~\\
\section{Introduction and Terminologies}
Knot Floer homology is an invariant for null-homologous knots $K$ in an oriented three-manifold $Y.$ The invariant was discovered independently by Ozsv\'ath-Szab\'o [10] and Rasmussen [13]. There are several versions of knot Floer homology. In this paper we focus on the simplest version, which is denoted as $\widehat{HFK}$ and takes the form of a bigraded, finite dimensional vector space over $\mathbb{F}=\mathbb{F}_2.$

We can naturally extend the definition of $\widehat{HFK}$ to oriented, null-homologous links in $Y,$ since oriented, null-homologous links in $Y$ correspond to certain oriented, null-homologous knots in $Y\#^{c-1}(S^2\times S^1)$ ($c$ is the number of components of the original link). For this paper, we concern with the special case when the ambient manifold $Y$ is $S^3.$ 

Let $L$ be a link (possibly a knot), there is a finite integer $u(L)$ such that $u(L)=\max\{s\in \mathbb{Z}:\widehat{HFK}(L,s)\not=0\}.$ For positive braid links, both the top term and the second-to-top term of their $\widehat{HFK}$ are known. Stallings [15] showed that non-split positive braid links are fibered, and Ozsv\'ath-Szab\'o [11] showed that a fibered link $L$ has $\widehat{HFK}(L,u(L))\cong \mathbb{F}$. The second-to-top term is computed in a paper by Zhechi Cheng [3]. Cheng's result is that for a positive braid link $L,$ $$\widehat{HFK}(L,u(L)-1)\cong \mathbb{F}^{p(L)+|L|-s(L)}[-1]\otimes(\mathbb{F}[0]\oplus \mathbb{F}[-1])^{\otimes s(L)-1}.$$ Here, $|L|$ is number of components of $L,$ $s(L)$ is the number of split factors of $L,$ and $p(L)$ is the number of prime factors of $L.$ Specifically, $p(L)$ is defined by $p(L_1\sqcup \cdots \sqcup L_{s(L)})=p(L_1)+\cdots+p(L_{s(L)})$ and $p(L_i)$ is the largest possible number of components splitting $L_i$ into connected sums and $p(unknot)=0.$
    
In a 2009 paper, Ni computed the top term of $\widehat{HFK}$ of closed 3-braids, which are not necessarily positive [8]. In this paper, we compute the second-to-top term of $\widehat{HFK}$ of closed 3-braids, using the exact triangle and other techniques. Note that we use the same convention for Maslov gradings as in Cheng's paper, which always takes integral values. Our convention differs from the convention in [10], which can take half-integral values, by a $+\frac{|L|-1}{2}$ shift.

Let $\sigma_1,\sigma_2$ be the standard Artin generators of the group of 3-braids $B_3.$ Let $a_1=\sigma_1,$ $a_2=\sigma_2,$ $a_3=\sigma_2\sigma_1\sigma_2^{-1}$. $B_3$ can be presented by $<a_1,a_2,a_3:a_2a_1=a_3a_2=a_1a_3>.$ 
\begin{figure}[H]
 \centering
  \includegraphics[width=0.75\textwidth]{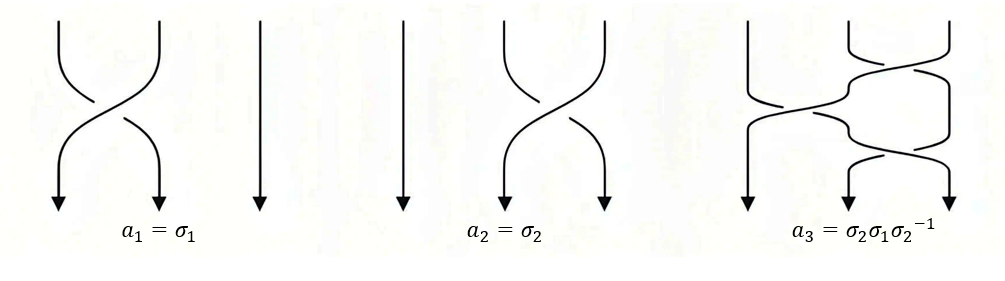}
 \caption{Generators of $B_3$}
 \end{figure}

For $w\in B_3,$ denote the closure of $w$ by $cl(w),$ and let $UT(w)$ be the reduced word of $w$ obtained as follows: if the subword $a_ia_i$ is contained in $w$, replace it with a single $a_i;$ repeat until $w$ does not contain subword of the form $a_ia_i.$

We do the computation case-by-case according to a full classification of closed 3-braids by Xu [16] (see Section 2 for details). The result is summarized in the following theorem:
\begin{theorem} For $w\in B_3,$ let $L$ be $cl(w).$ Let $\zeta(w)$ be the absolute value of the coefficient of the second-to-top term of $(t^{\frac{-1}{2}}-t^{\frac{1}{2}})^{|L|-1}\Delta_{L}(t)$. $\widehat{HFK}(L,u(L)-1)$ is as follows:
\begin{table}[H]
\centering
\caption{Classification of $\widehat{HFK}(L,u(L)-1)$ for closed 3-braids}
\setlength{\tabcolsep}{6pt}
\renewcommand{\arraystretch}{1.30}
\begin{tabular}{|l|p{0.40\linewidth}|p{0.46\linewidth}|}
\hline
\textbf{Type} & \textbf{Subtype (representative word $w$)} & $\widehat{HFK}(L,u(L)-1)$ \\
\hline
\multirow{4}{*}{$\alpha^d P$} 
& $d>1$ or $l(P)=0$. & $\mathbb{F}^{\zeta(w)}[-1]$. \\ \cline{2-3}
& $d=1$, $l(UT(P))\equiv 1$ mod 3. & $\mathbb{F}^{\zeta(w)}[-1]$. \\ \cline{2-3}
& $d=1$, $l(UT(P))\equiv 2$ mod 3, $k=\lfloor\frac{l(UT(P))}{3}\rfloor$. 
& $\mathbb{F}[k-1]\ \oplus\
\begin{cases}
\mathbb{F}^{\zeta(w)+1}[-1], & k\ \text{odd},\\
\mathbb{F}^{\zeta(w)-1}[-1], & k\ \text{even}.
\end{cases}$ \\ \cline{2-3}
& $d=1$, $l(UT(P))\equiv 0$ mod 3, $k=\frac{l(UT(P))}{3}>0$. 
& $\mathbb{F}[k-2]\ \oplus\
\begin{cases}
\mathbb{F}^{\zeta(w)+1}[-1], & k\ \text{even},\\
\mathbb{F}^{\zeta(w)-1}[-1], & k\ \text{odd}.
\end{cases}$\\
\hline
\multirow{3}{*}{$NP$} 
& General: $l(N),l(P)>1$, with top term $\widehat{HFK}(L,u(L))\cong \mathbb{F}[p]$ \newline($p$ determined by Proposition 3.5). 
& $\mathbb{F}^{\zeta(w)}[p-1]$. \\ \cline{2-3}
& $l(N)=1,$ $k=\lfloor\frac{l(UT(P))}{3}\rfloor>1$.
& $\mathbb{F}[0]\ \oplus\
\begin{cases}
\mathbb{F}^{\zeta(w)+1}[k-1], & k\ \text{even},\\
\mathbb{F}^{\zeta(w)-1}[k-1], & k\ \text{odd}.
\end{cases}$ \\ \cline{2-3}
& $l(N)=1$, $k=\lfloor\frac{l(UT(P))}{3}\rfloor\leq 1$.& $\mathbb{F}^{\zeta(w)}[k-1].$ \\ \cline{2-3}
\hline
\multirow{5}{*}{$P$} 
& $w\sim a_1^{n_1} a_2^{m_1} a_3^{\ell_1}$. 
& $\mathbb{F}^{\zeta(w)}[-1]$. \\ \cline{2-3}
& $w\sim a_1^{n_1} a_2^{m_1} a_3^{\ell_1} a_1^{n_2} a_2^{m_2} a_3^{\ell_2}$, each $n_i,m_i,l_i\leq 2$. & $\mathbb{F}^{|L|}[-1]\ \oplus\ \mathbb{F}^{|L|+\zeta(w)}[0]$. \\ \cline{2-3}
& $w\sim a_1^{n_1} a_2^{m_1} a_3^{\ell_1} a_1^{n_2} a_2^{m_2} a_3^{\ell_2}$ with $n_1>2$. 
& $\mathbb{F}^{\zeta(w_+)+(|L|-|L_+|)}[-1]\ \oplus\ \mathbb{F}^{\zeta(w_-)+(|L|-|L_+|)}[0]$.
\newline($w^+=\alpha^2a_3^{n_1-3}a_1^{-n_1}w$, $w^-=a_2^{-1}a_1^{2-n_1}wa_1^{n_1-2}$, $L_+$ is $cl(w_+)$)\\ \cline{2-3}
& $w\sim a_1^{n_1}\cdots a_3^{\ell_k}$ with $k>2$. 
& If $k$ even: $\mathbb{F}^{1+\zeta(w_+)+(|L|-|L_+|)}[-1]\ \oplus\ \mathbb{F}^{1+\zeta(w_-)+(|L|-|L_+|)}[k-2]$.\\[-2pt]
&& If $k$ odd: $\mathbb{F}^{\zeta(w_+)+(|L|-|L_+|)-1}[-1]\ \oplus\ \mathbb{F}^{\zeta(w_-)+(|L|-|L_+|)-1}[k-2]$. 
\newline($w^+=a_2w$, $w_-=a_2^{-1}w$, $L_+$ is $cl(w^+)$)\\ \cline{2-3}
& $w\sim a_1^{n}$ ($n\geq 0$). 
&If $n\leq 1:$ trivial.\\[-2pt]
&& If $n>1,$ $n$ odd: $\mathbb{F}[-1]\oplus \mathbb{F}[-2]$.
\newline If $n>1,$ $n$ even: $\mathbb{F}^2[-1]\oplus \mathbb{F}^2[-2]$. \\ 
\hline
Other 
&$w^{-1}\sim\alpha^{d}P$, $a_1^{n}$ or $a_2^{-1}P$. 
& Use symmetry: 
\newline$\widehat{HFK}_{m}(L,u(L)-1)\cong \widehat{HFK}_{-m+|L|-1}(\tilde{L},1-u(\tilde{L}))\cong \widehat{HFK}_{-m+|L|-3+2u(\tilde{L})}(\tilde{L},u(\tilde{L})-1)$. ( $\tilde{L}$ is $cl(w^{-1})$). \\
\hline
\end{tabular}
\label{tab:3braids-second-top}
\end{table}
\end{theorem}
\begin{remark}We omit certain cases in $P$ category because they are conjugate to other cases covered by the table. The case $l(UT(P))\equiv 1$ mod 3 is conjugate to the case $l(UT(P))\equiv 0$ mod 3, and the case $l(UT(P))\equiv 2$ mod 3 is conjugate to the case $\alpha^d P.$ The proof of Proposition 4.11 will show that we have indeed covered all cases in the $P$ category.\end{remark}
    For a few boundary cases, we check the $\widehat{HFK}$ on \textit{KnotInfo} [5].

    In fact, our results verify the following graded version of the rank inequality conjectured by Sivek (see [1]):
    \begin{corollary} For a closed 3-braid $L$ with $u(L)>0$, $rank(\widehat{HFK}_{m-1}(L,u(L)-1))\geq rank(\widehat{HFK}_m(L,u(L))).$
\end{corollary}
In [9], Ni showed that the above inequality holds if $L$ is a knot and $\widehat{HFK}(L,u(L))$ is supported in a single Maslov grading. For the remaining cases, we need to combine our Theorem 1.1 with the result in [2], which showed that $rank(\widehat{HFK}(L),u(L)-1)\geq |L|$ when $L$ is fibered.

If the equality holds, $L$ must be a knot. Examples when the equality holds include $L=cl(\alpha a_1a_2a_3a_1)=5_1$ and $L=cl(\alpha^d)=T(3,d),$ where $d>1,$ $3\nmid d$.
    
    \section{Preliminaries}
    The main tool used in Cheng's paper is the exact triangle introduced by Ozsv\'ath and Szab\'o [12]. This exact triangle describes the relationship between Floer homology of links related by skein relation, which is illustrated in Figure 2.
    \begin{figure}[H]
 \centering
  \includegraphics[width=0.5\textwidth]{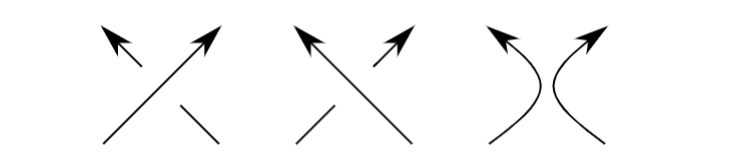}
 \caption{The skein relation, with $L_+$ $L_-$ and $L_0$ from left to right.}
 \end{figure}
    Namely, we have the following proposition [12]:
    \begin{proposition} There is an exact sequence
    $$\cdots\to \widehat{HFK}_m(L_+,s)\to \widehat{HFK}_m(L_-,s)\to \widehat{HFK}_{m-1}(L_0,s)\to \widehat{HFK}_{m-1}(L_+,s)\to\cdots$$ if $L_0$ has more components than $L_+$, and if $L_0$ has fewer components than $L_+$ there is an exact sequence
    $$\cdots\to \widehat{HFK}_m(L_+,s)\to \widehat{HFK}_m(L_-,s)\to (\widehat{HFK}(L_0)\otimes J)_{m-1,s}\to \widehat{HFK}_{m-1}(L_+,s)\to\cdots$$
    Here $J\cong \mathbb{F}[0,1]\oplus \mathbb{F}^2[-1,0]\oplus \mathbb{F}[-2,-1]$, where $[i,j]$ stands for Maslov and Alexander grading respectively.
    \end{proposition}
Hereafter we let $H_m(L_0,s)\cong\widehat{HFK}_m(L_0,s)$ if $|L_0|>|L_+|$ and $H_m(L_0,s)\cong(\widehat{HFK}(L_0)\otimes J)_{m,s}$ if $|L_0|<|L_+|.$ Also, $\overline{u}(L_0)=u(L_0)$ if $|L_0|>|L_+|$ and $\overline{u}(L_0)=u(L_0)+1$ if $|L_0|<|L_+|.$

The exact sequence can provide much information about $\widehat{HFK}$ of a knot (or link), if we know what the top Alexander grading $u(L)$ is. In fact, Ni [7] showed that $u(L)=\frac{|L|-\chi(L)}{2}$; when $L$ is a knot this coincides with the genus of $L.$ It is generally not very easy to know the exact value of $\chi(L)$, but for a closed 3-braid $L=cl(w)$, Xu's work [16] tells us that $\chi(L)$ is related to the word length $l(w)$. According to Xu, we have the following classification of closed 3-braids:
\begin{proposition} Let $\alpha=a_2a_1=a_3a_2=a_1a_3$. Every conjugacy class in $B_3$ can be represented as a shortest word in $a_1,a_2,a_3$ which is unique up to symmetries, such that the word has one of the following forms:
\\(i) $\alpha^d P$;
\\(ii) $N\alpha^{-d}$;
\\(iii) $NP$.
\\Here $d\geq 0$, $N^{-1}$ and $P$ are nondecreasing positive words, $P$ or $N$ may be empty. A positive word $a_{\epsilon_1}\cdots a_{\epsilon_n}$ is nondecreasing if for each $1\leq j<n,$ $\epsilon_{j+1}=\epsilon_{j}$ or $\epsilon_j+1$, where the subscript for $a$ is understood cyclically.
\end{proposition}
Xu's results also showed that for a shortest word $w$ as above, Euler characteristic of $cl(w)$ is $3-l(w)$. This means decreasing the word length might also decrease the genus of the corresponding link. When the shortest possible word length (up to conjugation) of a link in the skein triple is small enough, the skein exact sequence can result in an isomorphism between the second-to-top terms of the other two links in the triple, and the isomorphism shifts the Maslov grading in a fixed way. This isomorphism allows us to reduce a closed 3-braid to a braid whose second-to-top term is known.

Note that Xu's convention for $\sigma_1,\sigma_2$ is different from the one used in our paper. $\sigma_1,\sigma_2$ in [16] should be $\sigma_1^{-1},\sigma_2^{-1}$ under the convention of this paper. We change Xu's convention in order to use the skein exact triangle more conveniently. Xu's statements still hold under the new convention, because if $L$ is a closed 3-braid, and $\overline{L}$ is the mirror of $L$ obtained by switching over- and under-crossings in a projection for $L$, then $u(L)=u(\overline{L})$ [10].

An oriented link $L$ is called strongly quasi-positive when $L$ is the closure of a word $\beta\in B_n$ such that $\beta$ is the product of $\sigma_{i,j}=(\sigma_i\cdots\sigma_{j-2})\sigma_{j-1}(\sigma_i\cdots\sigma_{j-2})^{-1}$ [14]. Note that if $w\in B_3$ is positive in the generators $a_1,a_2,a_3$, then $cl(w)$ is strongly quasi-positive.

The following identities will be important in our computation:
\begin{align}
    a_{i}\alpha&=\alpha a_{i+1}\\
    a_i^{-1}a_j&=a_{i+2}a_{j+2}^{-1}\\
    a_{i+1}a_ia_{i+1}&=a_ia_{i+1}a_i\\
    \alpha a_i^{-1}&=a_{i+2}^{-1}\alpha=a_{i+1}\\
    \alpha^{-1}a_i&=a_{i+1}\alpha^{-1}=a_{i+2}^{-1}
\end{align}
    \section{Computing Maslov Grading of the Top Term}
    In [8], Ni has shown that the top term of $\widehat{HFK}$ of $cl(w)$ has rank 1 except when (i) $w$ is 0; (2) $w$ (or $w^{-1}$) is of type $P,$ and (up to conjugation) either starts with $a_1$ and ends with $a_3$ or is a power of $a_1$. However, to compute the second-to-top term, we must not only know the rank of the top term but also know the explicit Maslov grading. To describe the top term explicitly as a graded group, we need the following lemma:
    \begin{lemma}
        Consider a skein triple $L_0,L_-,L_+.$ We have
            \begin{align*}&(1)\quad\widehat{HFK}_m(L_+,u(L_+))\cong \widehat{HFK}_m(L_-,u(L_-))\text{ if $u(L_+)=u(L_-)>\overline{u}(L_0)$},\\
            &(2)\quad \widehat{HFK}_m(L_+,u(L_+))\cong H_m(L_0,\overline{u}(L_0))\cong \widehat{HFK}_m(L_0,u(L_0))\text{ if $u(L_+)=\overline{u}(L_0)>u(L_-)$},\\
            &(3)\quad \widehat{HFK}_m(L_-,u(L_+))\cong H_{m-1}(L_0,\overline{u}(L_0))\cong \widehat{HFK}_{m-1}(L_0,u(L_0))\text{ if $u(L_-)=\overline{u}(L_0)>u(L_+)$}.\end{align*}
    \end{lemma}
    \begin{proof}
        If $u(L_+)=u(L_-)>\overline{u}(L_0),$ then $H(L_0,u(L_+))$ is trivial. From the exact sequences $$\cdots\to \widehat{HFK}_m(L_+,s)\to \widehat{HFK}_m(L_-,s)\to H_{m-1}(L_0,s)\to \widehat{HFK}_{m-1}(L_+,s)\to\cdots$$
        we have $\widehat{HFK}_m(L_+,u(L_+))\cong \widehat{HFK}_m(L_-,u(L_-)).$

        Statements (2) and (3) can be proved with similar reasoning.
    \end{proof}
    Using the above lemma, we have the following propositions about the top term:
\begin{proposition}Suppose $w=\alpha^d P$ is a word in Xu's form, $d>0,$ $L=cl(w).$ Then, $\widehat{HFK}(L,u(L))\cong\mathbb{F}[0]$.
\end{proposition}
\begin{proof}
We prove by induction on the length of $P.$ When the length $l(P)=0,$ $L=cl(\alpha^d)$ and hence is a torus link; thus $\widehat{HFK}(L,u(L))\cong\mathbb{F}[0]$.

Assume $\widehat{HFK}(L,u(L))\cong\mathbb{F}[0]$ whenever $0\leq l(P)<n.$ Then, when $l(P)=n:$

Assume $P$ ends in $a_1.$ Then $w=a_1a_3 \alpha^{d-1}P'a_1$. Let $L_+=L,$ $L_0=cl(a_1a_3\alpha^{d-1}P'),$ $L_-=cl(a_1a_3\alpha^{d-1}P'a_1^{-1}).$ 

Since $a_1a_3\alpha^{d-1}P'a_1^{-1}\sim a_3\alpha^{d-1}P'$, we see that $u(L_+)=\overline{u}(L_0)>u(L_-).$ Thus, by Lemma 3.1 $\widehat{HFK}(L_+,u(L_+))\cong \widehat{HFK}(L_0,u(L_0))\cong \mathbb{F}[0]$.
\end{proof}

\begin{proposition}
Suppose $a_2^{-1}P$ is a shortest word in Xu's form and $l(P)>0$, which means $P$ cannot start or end with $a_2.$ Let $L=cl(a_2^{-1}P)$, and $k=\lfloor\frac{l(UT(P))}{3}\rfloor.$ Then, $\widehat{HFK}(L,u(L))\cong\mathbb{F}[k]$.
\end{proposition}
\begin{proof}
Assume $P=P_1a_j^2P_2$ for some $j\in \{1,2,3\}$, where $P_1a_j,a_jP_2$ are nondecreasing positive words. Let $L_+=L,$ $L_0=cl(a_2^{-1}P_1a_jP_2),$ $L_-=cl(a_2^{-1}P_1P_2).$ By Lemma 3.1, $\widehat{HFK}_m(L_+,u(L_+))\cong \widehat{HFK}_m(L_0,u(L_0))$.

Hence, we reduce to the case when $P=UT(P).$

Also, because $a_2^{-1}a_3^n= a_1^na_2^{-1},$ by conjugation we may assume $P$ starts with $a_1.$

If $k=0,$ we only need to consider the case $a_2^{-1}P=a_2^{-1}a_1$. In this case $L$ is the unknot, so $\widehat{HFK}(L,u(L))\cong \mathbb{F}[0]$.

Thus, the proposition holds if $k=0.$ Suppose $\widehat{HFK}(L,u(L))\cong \mathbb{F}[k]$ whenever $0\leq k<K.$ Then, when $k=K:$

If $P=P'a_1a_2a_3,$ let $L_+=cl(a^{-1}_2P' a_1a_2a_3),$ $L_0=cl(a^{-1}_2 P' a_1a_3),$ $L_-=cl(a^{-1}_2P' a_1a_2^{-1}a_3).$ 

Since $a^{-1}_2 P' a_1a_3\sim P'a_3$, by Lemma 3.1, $\widehat{HFK}_m(L_+,u(L_+))\cong \widehat{HFK}_m(L_-,u(L_-))$.

Also, $a^{-1}_2P' a_1a_2^{-1}a_3=a^{-1}_2P' a_1^2a_2^{-1}\sim a^{-2}_2P' a_1^2$. By Lemma 3.1 and the inductive hypothesis, we have $$\widehat{HFK}_m(L_-,u(L_-))\cong \widehat{HFK}_{m-1}(cl(a^{-1}_2P' a_1^2),u(cl(a^{-1}_2P' a_1^2)),\quad\widehat{HFK}(L_-,u(L_-))\cong \mathbb{F}[k].$$

If $P=P'a_1=P''a_3a_1$, let $L_+=cl(a^{-1}_2P' a_1),$ $L_0=cl(a^{-1}_2P'' a_3),$ $L_-=cl(a^{-1}_2 P''a_3a^{-1}_1).$ 

We have $a^{-1}_2 P''a_3a^{-1}_1\sim P''a_3\alpha^{-1}=P''a^{-1}_1$, so $u(L_-)<u(L_+)=\overline{u}(L_0)$. By Lemma 3.1, $\widehat{HFK}_{m}(L_+,u(L_+))\cong \widehat{HFK}_m(L_0,u(L_0)),$ which means $\widehat{HFK}(L_+,u(L_+))\cong \mathbb{F}[k]$ by discussion of the case when $P$ ends with $a_3.$

All in all, $\widehat{HFK}(L,u(L))\cong \mathbb{F}[k]$. We have proved the proposition by induction.\end{proof}

\begin{proposition} Suppose $L=cl(P)$, $P$ starts with $a_1$ and ends with $a_3$, $k=\frac{l(UT(P))}{3}>0$. Then, $\widehat{HFK}(L,u(L))\cong\mathbb{F}[0]\oplus \mathbb{F}[k-1]$.
\end{proposition}
\begin{proof}
Let $L_0=L,$ $L_+=cl(a_2P)$, $L_-=cl(a_2^{-1}P)$. By analyzing word length, we have $u(L_+)=u(L_-)=\overline{u}(L_0)$. Hence, there is an exact triangle among the top terms of $\widehat{HFK}(L_+),$ $H(L_0),$ $\widehat{HFK}(L_-).$

Set $u=u(L_+)$. By Propositions 3.2 and 3.3 we know that $\widehat{HFK}(L_+,u)\cong\mathbb{F}[0]$ and $\widehat{HFK}(L_-,u)\cong\mathbb{F}[k]$.

Consider the exact triangle:

\[
\begin{tikzcd}[row sep=large, column sep=large]
& H(L_{0}, u) \arrow[dr, "{[+0]}"] & \\
\widehat{HFK}(L_{-}, u) \arrow[ur, "{[-1]}"] \arrow[rr, "{[+0]}" above, leftarrow] & & \widehat{HFK}{(L_+, u)}
\end{tikzcd}
\]

The $[-1]$ and $[+0]$ labels mean that the map shifts the Maslov gradings by $-1$ and $+0$ respectively.

From the exact triangle, we see that $\widehat{HFK}(L,u(L))\cong H(L_0,u)\cong\mathbb{F}[0]\oplus \mathbb{F}[k-1].$
\end{proof}
\begin{proposition}Suppose $w=NP$ is a shortest word, $l(N)>0,$ $l(P)>0,$ $L=cl(w).$ Suppose $P$ ends with $a_1.$ Then, $\widehat{HFK}(L,u(L))=\mathbb{F}[\lfloor\frac{l(UT(P))-2s}{3}\rfloor+l(N)-1]$, where $$s=\begin{cases}
    \lfloor\frac{l(UT(N))-1}{2}\rfloor&\quad \text{if $N$ starts with $a_3^{-1}$},\\
    \lfloor\frac{l(UT(N))}{2}\rfloor&\quad \text{if $N$ starts with $a_2^{-1}$}.
\end{cases}$$
\end{proposition}
\begin{proof}
Assume $N=N'a_j^{-2}N''$ for some $j\in \{1,2,3\}$, where $a_j N'^{-1},{N''}^{-1}a_j$ are nondecreasing positive words. Let $L_-=L,$ $L_0=cl(N'a_j^{-1}N''P),$ $L_+=cl(N'N''P).$

Since $u(L_+)<u(L_-)=\overline{u}(L_0),$ we have $\widehat{HFK}_m(L_-,u(L_-))\cong \widehat{HFK}_{m-1}(L_0,u(L_0))$ by Lemma 3.1. This means Proposition 3.5 holds for $L_-$ if and only if it holds for $L_0.$ 

Hence, we reduce to the case when $N=UT(N)$.

The case when $l(N)=1$ is as in Proposition 3.3. Assume the proposition holds whenever $N=UT(N)$ and $1\leq l(N)< K.$ Then, when $N=UT(N)$ and $l(N)=K:$

Assume $P=P'a_1.$

If $N=a_3^{-1}a_2^{-1}N^*,$ let $L_-=L,$ $L_0=cl(a_2^{-1}N^*P),$ $L_+=cl(a_3a_2^{-1}N^*P).$

We have $a_3a_2^{-1}N^*P\sim \alpha a_2^{-1}N^* P'=a_3 N^* P'.$ Hence, $u(L_+)<u(L_-)=\overline{u}(L_0),$ and $\widehat{HFK}_m(L_-,u(L_-))\cong \widehat{HFK}_{m-1}(L_0,u(L_0))$ by Lemma 3.1. By inductive hypothesis, $\widehat{HFK}(L_0,u(L_0))\cong \mathbb{F}[\lfloor\frac{l(UT(P))-2s}{3}\rfloor+l(N)-2],$ where $s=\lfloor\frac{l(a_2^{-1}N^*)}{2}\rfloor=\lfloor\frac{l(N)-1}{2}\rfloor$. Thus $\widehat{HFK}(L_-,u(L_-))\cong \mathbb{F}[\lfloor\frac{l(UT(P))-2s}{3}\rfloor+l(N)-1],$ which is consistent with our proposition.

If $N=a_2^{-1}a_1^{-1}N^*,$ let $L_-=L,$ $L_0=cl(a_1^{-1}N^*P),$ $L_+=cl(a_2a_1^{-1}N^*P).$

We have $a_1^{-1}N^*P\sim N^*P'.$ Hence, $\overline{u}(L_0)<u(L_-)=u(L_+),$ and $\widehat{HFK}_m(L_-,u(L_-))\cong \widehat{HFK}_{m}(L_+,u(L_+))$ by Lemma 3.1. Since $a_2a_1^{-1}N^*P\sim a_1^{-1}N^*Pa_2$ and $B_3$ has rotation symmetry, we have $\widehat{HFK}(L_+,u(L_+))\cong \mathbb{F}[\lfloor\frac{l(UT(P))+1-2s}{3}\rfloor+l(N)-2]$ by inductive hypothesis, where $s=\lfloor\frac{l(a_1^{-1}N^*)-1}{2}\rfloor=\lfloor\frac{l(N)}{2}\rfloor-1$. Let $s'=s+1,$ then $\widehat{HFK}(L_-,u(L_-))\cong \mathbb{F}[\lfloor\frac{l(UT(P))+1-2s'+2}{3}\rfloor+l(N)-2]\cong \mathbb{F}[\lfloor\frac{l(UT(P))-2s'}{3}\rfloor+l(N)-1],$ which is consistent with our proposition.

Thus, we have proved the proposition by induction.
\end{proof}
\section{Computing Maslov Grading of the Second-to-top Term}
\subsection{Preliminary Lemmas}
In this section, we analyze at which Maslov gradings are the second-to-top terms non-trivial.

We need the following lemmas:
\begin{lemma}
        Consider a skein triple $L_0,L_-,L_+.$ We have \begin{align*}
            &(1)\quad \widehat{HFK}_m(L_+,u(L_+)-1)\cong \widehat{HFK}_m(L_-,u(L_-)-1)\text{ if $u(L_+)=u(L_-)>\overline{u}(L_0)+1$},\\
            &(2)\quad\widehat{HFK}_m(L_+,u(L_+)-1)\cong H_m(L_0,\overline{u}(L_0)-1)\text{ if $u(L_+)=\overline{u}(L_0)>u(L_-)+1$},\\
            &(3)\quad\widehat{HFK}_m(L_-,u(L_+)-1)\cong H_{m-1}(L_0,\overline{u}(L_0)-1)\text{ if $u(L_-)=\overline{u}(L_0)>u(L_+)+1$}.
        \end{align*}
    \end{lemma}
    \begin{proof}
        If $u(L_+)=u(L_-)>\overline{u}(L_0)+1,$ then $H(L_0,u(L_+)-1)$ is trivial. From the exact sequences $$\cdots\to \widehat{HFK}_m(L_+,s)\to \widehat{HFK}_m(L_-,s)\to H_{m-1}(L_0,s)\to \widehat{HFK}_{m-1}(L_+,s)\to\cdots$$
        we have $\widehat{HFK}_m(L_+,u(L_+)-1)\cong \widehat{HFK}_m(L_-,u(L_-)-1).$

        Statements (2) and (3) can be proved with similar reasoning.
    \end{proof}
    \begin{lemma}
        Consider a skein triple $L_0,L_-,L_+.$ If $L_-$ is fibered and strongly quasi-positive, then when $m\not=-1,$ $\widehat{HFK}_m(L_+,u(L_+)-1)\cong H_m(L_0,\overline{u}(L_0)-1)$.
    \end{lemma}
    \begin{proof}
        Set $u=u(L_+)$. Consider the exact sequences
        $$\cdots \widehat{HFK}_m(L_+,u-1)\xrightarrow{F_m}\widehat{HFK}(L_-,u-1)\xrightarrow{G_m}H_{m-1}(L_0,u-1)\to\cdots$$
        Since $L_-$ is strongly quasi-positive, $\tau(L_-)=u(L_-)=u-1$ [4], and $\widehat{HFK}(L_-,u-1)\cong \mathbb{F}[0]$ by Section 3. Hence, by the reasoning in Corollary 3.4 of Cheng's paper [3], $G_0$ is injective and hence $F_0$ is 0 by exactness. Thus, we have the exact sequence $0\to H_0(L_0,u-1)\to \widehat{HFK}_{0}(L_+,u-1)\to 0$, which means $\widehat{HFK}_0(L_+,u-1)\cong H_0(L_0,u-1)$. The exact sequences also yield that $\widehat{HFK}_m(L_+,u-1)\cong H_m(L_0,u-1)$ when $m\not=-1,0.$

        Hence, $\widehat{HFK}_m(L_+,u(L_+)-1)\cong H_m(L_0,\overline{u}(L_0)-1)$ when $m\not=-1.$
    \end{proof}
Now, we could do case-by-case discussion according to Xu's classification.
\subsection{Case 1: Suppose $w=\alpha^d P$ is a word in Xu's form, $d>0,$ $L=cl(w).$}

If $l(P)=0,$ $w$ is a positive word and the second-to-top term is as in Cheng's paper, which means it is supported at Maslov grading $-1$ because torus link is non-split [3]. If $d>1,l(P)>0,$ suppose $P$ starts with $a_1$. Then, $w=\alpha^{d-1}a_2a_1^2P'.$ Let $L_+=L,$ $L_0=cl(\alpha^{d-1}a_2a_1P')$, $L_-=cl(\alpha^{d-1}a_2P').$ $u(L_-)=u(L_+)-1=\overline{u}(L_0)-1$, and $L_-$ is fibered and strongly quasi-positive, with top term $\mathbb{F}[0]$. By Lemma 4.2, $\widehat{HFK}_m(L_+,u(L_+)-1)\cong H_{m
}(L_0,\overline{u}(L_0)-1)$ for $m\not=-1$. We have thus shown by induction that if $d>1$, \begin{align}\widehat{HFK}(L,u(L)-1)\cong \mathbb{F}^t[-1]\text{ for some $t\geq 0$}\end{align}
\begin{proposition}
    Suppose that $d=1,$ $l(P)>0,$ $k=\lfloor\frac{l(UT(P))}{3}\rfloor\geq 0:$
\begin{subequations}\label{eq:UTcases}
\begin{align}
\widehat{HFK}(L,u(L)-1)
  &\cong \mathbb{F}^t[-1]\text{ for some $t\geq 0$}, &&\text{if } l(UT(P))\equiv 1\!\!\pmod{3}, \label{eq:case1} \\
\widehat{HFK}(L,u(L)-1)
  &\cong \mathbb{F}[k-1]\oplus\mathbb{F}^t[-1]\text{ for some $t\geq 0$}, &&\text{if } l(UT(P))\equiv 2\!\!\pmod{3}, \label{eq:case2}\\
\widehat{HFK}(L,u(L)-1)
  &\cong \mathbb{F}[k-2]\oplus\mathbb{F}^t[-1]\text{ for some $t\geq 0$}, &&\text{if } l(UT(P))\equiv 0\!\!\pmod{3}. \label{eq:case3}
\end{align}
\end{subequations}
\end{proposition}
\begin{proof}

If $P=a_1^{n}$ or $a_1^{n}a_2^{m},$ $n,m>0$, then $w=\alpha P=a_2a_1P$ is a positive braid word. Also, $w$ is fibered and hence non-split. Thus, $\widehat{HFK}(L,u(L)-1)$ is supported at Maslov grading -1 by Cheng's formula. By analyzing word length we see that $u(L)>0,$ so $L$ is not the unknot. Then, from Cheng's formula we see that the second-to-top term of $\widehat{HFK}(L,u(L)-1)$ has rank $p(L)+|L|-s(L)\geq p(L)>0.$

If $P=a_1^{n}a_2^{m}a_3^{l}=a_1^{n}a_2^{m}a_2a_1^{l}a_2^{-1},$ then $\alpha P=a_2a_1^{n+1}a_2^{m}a_2a_1^{l}a_2^{-1}\sim a_1^{n+1}a_2^{m}a_2a_1^{l}$ is a positive word. Since $P$ is fibered, $P$ is non-split. Hence, $\widehat{HFK}(L,u(L)-1)$ is supported at Maslov grading -1. Also, by analyzing the word length we could see that $u(L)>0,$ which means $L$ is not the unknot and hence $\widehat{HFK}(L,u(L)-1)$ has rank $>0.$

When $l(UT(P))>3:$

If $P=P_1a_j^2P_2$ for some $j\in \{1,2,3\},$ with $P_1a_j,a_jP_2$ nondecreasing positive words, consider the skein triple $L_+=L, L_0=cl(\alpha P_1a_jP_2), L_-=cl(\alpha P_1P_2).$ 

We have $u(L_-)=u(L_+)-1=\overline{u}(L_0)-1$, and $L_-$ is strongly quasi-positive. If $P_1P_2$ is a nondecreasing positive word, then $L_-$ is fibered. If $P_1P_2$ is not a nondecreasing positive word, then $P_1=P'_1a_{j-1}$ and $P_2=a_{j+1}P'_2,$ with $P'_1a_{j-1},a_{j+1}P'_2$ nondecreasing positive. By Lemma 2.7 of [8], $L_-$ is fibered if and only if $cl(\alpha P_1'a_{j-1}^2a_{j+1}^2P'_2)$ is fibered. By (1), $P_1'a_{j-1}^2a_{j+1}^2P'_2=P_1'a_{j-1}\alpha a_{j+1}P'_2=\alpha P'_3a_j a_{j+1} P'_2$, with $P'_3 a_j$ nondecreasing positive. This means $cl(\alpha P_1'a_{j-1}^2a_{j+1}^2P'_2)$ is fibered and hence $L_-$ is fibered. Hence, by Lemma 4.2, $\widehat{HFK}_m(L_+,u(L_+)-1)\cong H_m(L_0,\overline{u}(L_0)-1)$ when $m\not=-1.$

Thus, we reduce to the case $P=UT(P)$. During this process we only alters the Maslov grading -1 part of the second-to-top term of $\widehat{HFK}(L)$. 

If $l(UT(P))\equiv 1$ mod 3, let $P=a_1a_2a_3P'=a_1a_2a_3P''a_1$. By (3), we have $w=a_3a_1 a_2a_1a_3P'=a_3a_1 a_2a_2a_1P'.$ 

Now, let $L_+=L,$ $L_0=cl(a_3a_1 a_2a_1P'),$ $L_-=cl(a_3a_1a_1P').$

By (3), we have $a_3a_1 a_2a_1P'\sim a_1a_3a_1 a_2a_1P''=\alpha^2 a_2P''.$ Also, $a_3a_1a_1P'\sim \alpha a_1^2 P''.$ Hence, $u(L_-)=u(L_+)-1=\overline{u}(L_0)-1$. By (6), $H(L_0,\overline{u}(L_0)-1)$ is supported at Maslov grading -1. Moreover, $L_-$ is fibered by Proposition 3.2, and $L_-$ is strongly quasi-positive. By Lemma 4.2, $\widehat{HFK}(L,u(L)-1)$ is supported at Maslov grading -1.

If $l(UT(P))\equiv 2$ mod 3, let $P=P'a_1a_2$. Let $L_+=L,$ $L_0=cl(\alpha P'a_1),$ $L_-=cl(\alpha P'a_1a_2^{-1}).$ 

We have $\alpha P'a_1a_2^{-1}\sim a_1 P'a_1\sim a_1^{2}P'$, and $a_1^2P'$ ends with $a_3.$

Thus, $u(L_-)=u(L_+)-1=\overline{u}(L_0)-1.$ By Proposition 3.4, $\widehat{HFK}(L_-,u(L_-))$ is $\mathbb{F}[0]\oplus\mathbb{F}[k-1].$ By (7a), the second-to-top term of $\widehat{HFK}(L_0)$ is supported at Maslov grading -1.

Now, take $u=u(L_+)$ and consider the exact sequences $$\cdots \to H_{m}(L_0,u-1)\to\widehat{HFK}_{m}(L_+,u-1)\xrightarrow{F_m}\widehat{HFK}_{m}(L_-,u-1)\xrightarrow{G_m}H_{m-1}(L_0,u-1)\to\cdots$$

If $k=1,$ let $L'_+=cl(\alpha a_1^2a_2a_3a_1a_2),$ $L'_0=L_+,$ $L'_-=cl(\alpha a_2a_3a_1a_2)$. By Lemma 4.2, $H_m(L'_0,\overline{u}(L'_0)-1)\cong \widehat{HFK}_m(L'_+,u(L'_+)-1)$ when $m\not=-1$. In fact, $L'_+$ is just $10_{161}$ and \textit{KnotInfo} [5] shows that $\widehat{HFK}(L'_+,u(L'_+)-1)\cong\mathbb{F}[0]\oplus\mathbb{F}[-1]$. All in all, $\widehat{HFK}(L_+,u-1)\cong\mathbb{F}[0]\oplus\mathbb{F}^t[-1]$ for some $t\geq 0.$

If $k>1,$ since $H_{k-1}(L_0,u-1)\cong H_{k-2}(L_0,u-1)\cong 0$ by (7a), we have an exact sequence $$0\to \widehat{HFK}_{k-1}(L_+,u-1)\xrightarrow{F_{k-1}} \widehat{HFK}_{k-1}(L_-,u-1)\xrightarrow{G_{k-1}} 0,$$ which means $\widehat{HFK}_{k-1}(L_+,u-1)\cong\widehat{HFK}_{k-1}(L_-,u-1)\cong\mathbb{F}$. 

If $k>2,$ since $\widehat{HFK}_{k-2}(L_-,u-1)\cong 0,$ we have $0\to\widehat{HFK}_{k-2}(L_+,u-1)\to 0,$ so $\widehat{HFK}_{k-2}(L_+,u-1)\cong0.$ Also, $H_{0}(L_0,u-1)\cong0$ by (7a) and $G_{k-1}$ is 0. Yet $L_-$ is strongly quasi-positive, which means $\tau(L_-)=u-1,$ so some element in $\widehat{HFK}(L_-,u-1)$ is mapped to the top generator of $\widehat{HF}(\#^{|L|-1} S^1\times S^2) $ under map induced by inclusion $\iota:\widehat{CFK}(L_-,u-1)\to \widehat{CF}(\#^{|L|-1} S^1\times S^2)$. Following the reasoning in Corollary 3.4 of [3], this means the map $G:\widehat{HFK}(L_-,u-1)\to H(L_0,u-1)$ is nontrivial, and hence either $G_{k-1}$ or $G_0$ is nontrivial. That is, $G_0$ is nontrivial and therefore injective ($\widehat{HFK}_0(L_-,u-1)$ is 1-dimensional), and hence $F_0$ is 0, which gives rise to the exact sequence $0\to \widehat{HFK}_0(L_+,u-1)\to 0$. Hence, $\widehat{HFK}_0(L_+,u-1)\cong0$. Also, the exact sequences show that $\widehat{HFK}_m(L_+,u-1)\cong H_m(L_0,u-1)$ when $m\not=0,-1,k-2,k-1$. Thus, $\widehat{HFK}(L_+,u-1)$ is $\mathbb{F}[k-1]\oplus\mathbb{F}^t[-1]$ for some $t\geq 0$.

If $k=2,$ we have $\widehat{HFK}_{1}(L_+,u-1)\cong \mathbb{F}$ as above. $F_0$ is trivial by the same argument in the last paragraph, which gives rise to exact sequence $0\to \widehat{HFK}_0(L_+,u-1)\to 0$. Therefore, $\widehat{HFK}_0(L_+,u-1)$ is trivial. Also, the exact sequences show that $\widehat{HFK}_m(L_+,u-1)\cong H_m(L_0,u-1)$ when $m\not=0,-1,1$. Thus, $\widehat{HFK}(L_+,u-1)$ is $\mathbb{F}[1]\oplus\mathbb{F}^t[-1]$ for some $t\geq 0$.

If $l(UT(P))\equiv 0$ mod 3, let $P=a_1a_2a_3 P'=a_1a_2a_3 P''a_3.$ Let $L_+=L,$ $L_0=cl(a_2 a_1a_2a_3P')$, $L_-=cl(a_2a_2a_3P').$

We have $a_2a_2a_3P'\sim \alpha a_2a_3P''$. Thus, $L_-$ is fibered and strongly quasi-positive, and $u(L_-)=u(L_+)-1=\overline{u}(L_0)-1.$ By Lemma 4.2 and (7b), $\widehat{HFK}(L,u(L)-1)\cong\mathbb{F}[k-2]\oplus \mathbb{F}^t[-1]$ for some $t\geq 0$ because $\widehat{HFK}(L_0,u(L_0)-1)\cong\mathbb{F}[k-2]\oplus \mathbb{F}^{t'}[-1]$ for some $t'\geq 0$.

We have thus proved the proposition.\end{proof}
\subsection{Case 2: Suppose $w=NP,$ $l(N)>1,$ $l(P)>1,$ $L=cl(w).$}

We say that $L$ has property $(*)$ if $\widehat{HFK}(L,u(L))$ is $\mathbb{F}[p]$ and $\widehat{HFK}(L,u(L)-1)$ is supported at Maslov grading $p-1.$
\begin{proposition}
    If $|l(UT(N))-l(UT(P))|=0,$ then $L$ is homologically $\delta$-thin and hence has property $(*)$.
\end{proposition} 
\begin{proof}
First, assume $N$ starts with $a_2^{-1}$ and $P$ ends with $a_1.$

When $l(UT(N))=l(UT(P))=1,$ $NP=\sigma_{2}^{-p}\sigma_1^{q}$ ($p,q>0$), which is alternating and hence it is homologically $\delta$-thin [6].

Assume whenever $1\leq l(UT(N))=l(UT(P))<K$ and $N$ starts with $a_2^{-1}$ and $P$ ends with $a_1,$ we have $NP=
    \sigma_2^{-p_1}\sigma_1^{q_1}\cdots \sigma_2^{-p_k}\sigma_1^{q_k}$ or $\sigma_1^{q_0}\sigma_2^{-p_1}\sigma_1^{q_1}\cdots \sigma_2^{-p_k}\sigma_1^{q_k}.$ Here, each $p_i,q_i>0$. Then, when $l(UT(N))=l(UT(P))=K$:

We may suppose that $N=a_2^{-p}N'=a_2^{-p}a_1^{-1}N''$ and $P=P'a_1^q=P''a_3a_1^q,$ $p,q>0.$ Conjugating by $\alpha^2,$ $N'P'$ becomes a word starting with $a_2^{-1}$ and ending with $a_1.$ By inductive hypothesis, $\alpha^2 N'P'\alpha^{-2}=\sigma_2^{-p_1}\sigma_1^{q_1}\cdots \sigma_2^{-p_k}\sigma_1^{q_k}$ or $\sigma_1^{q_0}\sigma_2^{-p_1}\sigma_1^{q_1}\cdots \sigma_2^{-p_k}\sigma_1^{q_k}$. Thus, \begin{align*}
N'P'&=\alpha^3N'P'\alpha^{-3}\\&=\alpha\sigma_2^{-p_1}\sigma_1^{q_1}\cdots \sigma_2^{-p_k}\sigma_1^{q_k}\alpha^{-1}\text{ or }\alpha\sigma_1^{q_0}\sigma_2^{-p_1}\sigma_1^{q_1}\cdots \sigma_2^{-p_k}\sigma_1^{q_k}\alpha^{-1},\end{align*}
and
\begin{align*}NP&=a_2^{-p}\alpha\sigma_2^{-p_1}\sigma_1^{q_1}\cdots \sigma_2^{-p_k}\sigma_1^{q_k}\alpha^{-1}a_1^{q}\\&=\sigma_2^{-p+1}\sigma_1\sigma_2^{-p_1}\sigma_1^{q_1}\cdots \sigma_2^{-p_k}\sigma_1^{q_k-1}\sigma_2^{-1}\sigma_1^{q}\end{align*} or \begin{align*}
    NP&=a_2^{-p}\alpha\sigma_1^{q_0}\sigma_2^{-p_1}\sigma_1^{q_1}\cdots \sigma_2^{-p_k}\sigma_1^{q_k}\alpha^{-1}a_1^{q}\\&=\sigma_2^{-p+1}\sigma_1^{q_0+1}\sigma_2^{-p_1}\sigma_1^{q_1}\cdots \sigma_2^{-p_k}\sigma_1^{q_k-1}\sigma_2^{-1}\sigma_1^{q}.\end{align*}

We thus show by induction that if $N$ starts with $a_2^{-1}$ and $P$ ends with $a_1,$ $l(UT(N))=l(UT(P))$, then $NP=
    \sigma_2^{-p_1}\sigma_1^{q_1}\cdots \sigma_2^{-p_k}\sigma_1^{q_k}$ or $\sigma_1^{q_0}\sigma_2^{-p_1}\sigma_1^{q_1}\cdots \sigma_2^{-p_k}\sigma_1^{q_k}.$ This means $L$ is alternating and hence homologically $\delta$-thin.

The case when $N$ starts with $a_2^{-1}$ and $P$ ends with $a_3$ follows by mirror symmetry [10].

Hence $L$ has property $(*)$ if $l(UT(P))-l(UT(N))=0.$\end{proof}

Now, we may assume that $L$ has property $(*)$ whenever $l(P),l(N)\geq 2$ and $0\leq |l(UT(N))-l(UT(P))|<K.$

We have the following lemmas:
\begin{lemma}
    Consider a skein triple $L_0,L_-,L_+$. If $u(L_+)=u(L_-)>\overline{u}(L_0)+1$, then $L_+$ has property $(*)$ if and only if $L_-$ has property $(*)$. Analogous statements hold if $u(L_+)=\overline{u}(L_0)>u(L_-)+1$ or $u(L_-)=\overline{u}(L_0)>u(L_+)+1$.
\end{lemma}
\begin{proof}
    The lemma is immediate from Lemma 3.1 and Lemma 4.1.
\end{proof}
\begin{lemma}
     Consider a skein triple $L_0,L_-,L_+$. If $u(L_-)=u(L_+)-1=\overline{u}(L_0)-1$ and $L_-,L_+,L_0$ all have top term $\mathbb{F}[p],$ then $L_0$ has property $(*)$ if $L_+$ has property $(*).$
\end{lemma}
\begin{proof}

The lemma follows from the exact triangle below:
    \[
\begin{tikzcd}[row sep=large, column sep=large]
& H(L_{0}, \overline{u}(L_0)-1) \arrow[dr, "{[+0]}"] & \\
\widehat{HFK}(L_{-}, u(L_-)) \arrow[ur, "{[-1]}"] \arrow[rr, "{[+0]}" above, leftarrow] & & \widehat{HFK}{(L_+, u(L_+)-1)}
\end{tikzcd}
\]\end{proof}
Now we want to show that $L$ has property $(*)$ if $|l(UT(N))-l(UT(P))|=K.$
\begin{proposition}
    If $l(UT(P))-l(UT(N))=K$, $N$ starts with $a_2^{-1}$ and $P$ ends with $a_1,$ then $L$ has property $(*)$. 
\end{proposition}
\begin{proof}
Suppose $N=a_2^{-1}N'=a_2^{-2}N''$ or $a_2^{-1}a_1^{-1}N''$. If $P=P'a_1=P''a_1^2,$ let $L_-=cl(a_3^{-1}NP),$ $L_0=L,$ $L_+=cl(a_3NP)$.

By (4), we have $a_3NP\sim a_1^2a_3 N P''=a_3N''P''$ or $a_2N''P''$, so $u(L_+)+1<u(L_-)=\overline{u}(L_0)$. Also, by inductive hypothesis, $L_-$ has property $(*)$. Thus, by Lemma 4.5, $L_0$ has property $(*)$.

If $P=P'a_1=P''a_3a_1,$ let $L_+=cl(NPa_1),$ $L_0=L,$ $L_-=cl(NP')$. By Proposition 3.5, $\widehat{HFK}(L_+,u(L_+))\cong \widehat{HFK}(L_0,u(L_0))\cong \mathbb{F}[p]$ for some $p.$ Also, let $L'_+=L_0,$ $L'_0=L_-,$ $L'_-=cl(NP'a_1^{-1})$. Since $NP'a_1^{-1}\sim \alpha^{-1}N'P'\sim a_1^{-1} N'P''$, we have $u(L'_-)<u(L'_+)=\overline{u}(L'_0).$ By Lemma 3.1, $\widehat{HFK}(L'_0,u(L'_0))\cong \mathbb{F}[p].$ By discussion of the case $P=P''a_1^2,$ we see that $L_+$ has property $(*)$. Hence, $L_0$ has property $(*)$ by Lemma 4.6.
\end{proof}

Now, suppose $l(UT(N))>l(UT(P)).$
\begin{proposition}
    If $l(UT(N))-l(UT(P))=K$, $N$ starts with $a_2^{-1}$ and ends with $a_i^{-1}$, and $P$ starts with $a_{i+2}$ and ends with $a_1,$ then $L$ has property $(*)$.
\end{proposition}
\begin{proof}
If $N=a_2^{-1}N'a^{-2}_{i},$ $P=a_{i+2}P'a_1,$ let $L_0=L,$ $L_+=cl(Na_{i+1}P),$ $L_-=cl(Na_{i+1}^{-1}P)$.

By (5), we have \begin{align*}   Na_{i+1}^{-1}P&=\begin{cases}a_2^{-1}N'a^{-1}_{i+1}P''a_1 &\quad \text{if  }P'=a_{i+2}P'' \\a_2^{-1}N'a^{-1}_{i+2}P''a_1 &\quad \text{if }P'=a_{i}P''\end{cases}.
\end{align*} By analyzing word length we see that $u(L_-)<u(L_+)-1=\overline{u}(L_0)-1.$ By Lemma 4.5 and the inductive hypothesis, $L_0$ has property $(*).$

If $N=a_2^{-1}N'a_{i+1}^{-1}a^{-1}_{i},$ $P=a_{i+2}P'a_1,$ we first consider the case when $P'=a_{i+2}P''.$ 

Let $L_-=L,$ $L_0=cl(a_2^{-1}N'a_{i+1}^{-1}P)$, $L_+=cl(a_2^{-1}N'a_{i+1}^{-1}a_iP)$.

By (4), we have \begin{align*}a_2^{-1}N'a_{i+1}^{-1}a_iP&=\begin{cases}a_2^{-1}N''a_i P''a_1& \text{if } N'=N''a_{i+1}^{-1}\\a_2^{-1}N''a_{i+1} P''a_1 &\text{if } N'=N''a_{i+2}^{-1}\end{cases}.\end{align*} Then, we could see that $u(L_+)<u(L_-)-1=\overline{u}(L_0)-1.$ By Lemma 4.5 and the inductive hypothesis, $L_-$ has property $(*).$

If $P'=a_{i}P'',$ let $L_0=L,$ $L_+=cl(Na_{i+2}P),$ $L_-=cl(Na_{i+2}^{-1}P)$.

Since $Na_{i+2}^{-1}P=NP'a_1=a_2^{-1}N'a_{i+1}^{-1}P''a_1,$ we see that $u(L_-)<u(L_+)-1=\overline{u}(L_0)-1.$ By Lemma 4.5 and discussion of the case $P'=a_{i+2}P''$, $L_0$ has property $(*)$.\end{proof}
Now, in order to show that $L$ has property $(*)$ if $|l(UT(N))-l(UT(P))|=K$, $N$ starts with $a_2^{-1}$ and $P$ ends with $a_1$, it remains to consider the case when $l(UT(N))>l(UT(P))$, $N$ ends with $a_i^{-1}$ and $P$ starts with $a_{i+1}.$
\begin{proposition}
    $L$ has property $(*)$ if $|l(UT(N))-l(UT(P))|=K$, $N$ starts with $a_2^{-1}$ and $P$ ends with $a_1$.
\end{proposition}
\begin{proof}
Suppose $l(UT(N))>l(UT(P)),$ $N=a_2^{-1}N'a^{-1}_{i},$ $P=a_{i+1}P'a_1.$

If $N'=N''a_i^{-1}$ and $P'=a_{i+1}^2P'',$ then $NP=a_2^{-1}N''a_i^{-1}a_{i+2}a_i^{-1}a_{i+1}^2P''a_1$ by (2).

Let $L_-=L,$ $L_+=cl(a_2^{-1}N''a_i^{-1}a_{i+2}a_ia_{i+1}^2P''a_1)$, $L_0=cl(a_2^{-1}N''a_i^{-1}a_{i+2}a_{i+1}^2P''a_1).$

By (4), we have \begin{align*}a_2^{-1}N''a_i^{-1}a_{i+2}a_{i+1}^2P''a_1&=\begin{cases}a_2^{-1}N'''a_{i+2} P''a_1& \quad \text{if } N''= N'''a_i^{-1}\\a_2^{-1}N'''a_{i} P''a_1 &\quad \text{if } N''= N'''a_{i+1}^{-1}\end{cases}.\end{align*} By analyzing word length we see that $\overline{u}(L_0)<u(L_-)-1=u(L_+)-1$. By Lemma 4.5 and the inductive hypothesis, $L_-$ has property $(*).$

If $N'=N''a_{i+1}^{-1}$ and $P'=a_{i+1}^2P''$, let $L_0=L,$ $L_-=cl(Na_i^{-1}P),$ $L_+=cl(Na_i P).$ 

We have $Na_i P=a_2^{-1}N'P=a_2^{-1}N''P'a_1,$ so $u(L_+)<u(L_-)-1=\overline{u}(L_0)-1.$ By Lemma 4.5 and discussion of the case $N'=N''a_i^{-1}$, $L_0$ has property $(*)$.

That is, $L$ has property $(*)$ if $P'=a_{i+1}^2 P''$.

Now, if $P'=a_{i+1}a_{i+2}P'',$ let $L_0=L,$ $L_+=cl(Na_{i+1}Pa_1)$, $L_-=cl(N P^{'}a_1)$. By analyzing word length we see that $u(L_-)=u(L_+)-1$, $\overline{u}(L_0)=u(L_+).$ Set $u=u(L_+).$ By Proposition 3.5, we see that $\widehat{HFK}(L_+,u)\cong \widehat{HFK}(L_-,u)\cong \widehat{HFK}(L_0,u(L_0))\cong\mathbb{F}[p]$ for some $p$. By our discussion of the case $P'=a_{i+1}^2P''$, we see that $L_+$ has property $(*).$ Hence, $L_0$ has property $(*)$ by Lemma 4.6.

That is, $L$ has property $(*)$ if $P'=a_{i+1}a_{i+2}P''$.

If $P'=a_{i+2}P''$, let $L_0=L,$ $L_+=cl(Na_{i+1}Pa_1)$, $L_-=cl(N P^{'}a_1)$. We see that $u(L_-)=u(L_+)-1,$ $u(L_+)=\overline{u}(L_0)$. Set $u=u(L_+)$. By Lemma 3.1, $\widehat{HFK}(L_+,u)\cong\widehat{HFK}(L_0,u(L_0))\cong\mathbb{F}[p]$ for some $p.$ 

Also, let $L'_+=L_0,$ $L'_0=L_-,$ $L'_-=cl(N a_{i+1}^{-1}P^{'}a_1)$. By (5), $N a_{i+1}^{-1}P^{'}a_1=a_2^{-1}N'a_{i+1}^{-1}P''a_1$, we see that $u(L'_-)=u(L'_+)-1=\overline{u}(L'_0)-1.$ By Lemma 3.1, $\widehat{HFK}(L'_0,u(L'_0))\cong\widehat{HFK}(L_-,u)\cong\mathbb{F}[p]$.

By our discussion of the case $P'=a_{i+1}a_{i+2}P''$, we see that $L_+$ has property $(*).$ Hence, $L_0$ has property $(*)$ by Lemma 4.6.\end{proof}

All in all, we have shown by induction that $L$ has property $(*)$ if $N$ starts with $a_2^{-1}$ and $P$ ends with $a_1$. By mirror symmetry, similar property holds when $N$ starts with $a_2^{-1}$ and $P$ ends with $a_3$ [10].

\subsection{Case 3: Suppose $w=NP$, $1=l(N)\leq l(P),$ $L=cl(w)$.}

Without loss of generality suppose $N=a^{-1}_2.$ Because we require $w$ to be the shortest, $P$ cannot start or end with $a_2.$ Since $a_2^{-1}a_3^n=a_1^na_2^{-1},$ up to conjugation we may suppose $P$ starts with $a_1.$ Let $k=\lfloor\frac{l(UT(P))}{3}\rfloor$. 

If $k=0,$ $P$ is 0 or of the form $a_1^{n}$ with $n>0.$ If $P$ is 0 then $L$ splits into two unknots and hence the second-to-top term is trivial.

If $P$ is of the form $a_1^{n}$ with $n>0,$ then, $L=T(2,n)$ is alternating. Hence, $L$ is $\delta$-thin [6]. The top term of $\widehat{HFK}(L)$ is $\mathbb{F}[0]$ by Proposition 3.3, which means the second-to-top term is supported at Maslov grading -1.
\begin{proposition} We have

\begin{align}
\widehat{HFK}(L,u(L)-1)\cong\begin{cases}
    \mathbb{F}^t[0]\text{ for some $t\geq 0$}, &\text{if } k=1,\\
  \mathbb{F}[0]\oplus\mathbb{F}^t[k-1]\text{ for some $t\geq 0$}, &\text{if } k>1.
\end{cases}
\end{align}
\end{proposition}

\begin{proof}
Let $L_0=L,$ $L_+=cl(P),$ $L_-=cl(a_2^{-2}P).$ 

By conjugation, $P$ would start with $a_1$ and end with $a_3$. By analyzing word length we see that $u(L_+)=u(L_-)-1=\overline{u}(L_0)-1$. Set $u=u(L_-).$ By Propositions 3.4 and 3.5, and the discussion of Case 2, $\widehat{HFK}(L_+,u-1)\cong\mathbb{F}[0]\oplus\mathbb{F}[k-1]$, and $\widehat{HFK}(L_-,u-1)$ is supported at Maslov grading $k$.

Consider the exact triangle
\[
\begin{tikzcd}[row sep=large, column sep=large]
& H(L_{0}, u-1) \arrow[dr, "{[+0]}"] & \\
\widehat{HFK}(L_{-}, u-1)\arrow[ur, "{[-1]}"] \arrow[rr, "{[+0]}" above, leftarrow] & & \widehat{HFK}{(L_+, u-1)}
\end{tikzcd}
\]

The exact triangle reveals that $\widehat{HFK}(L_0,u(L_0)-1)\cong \mathbb{F}^t[0]$ for some $t\geq 0$ if $k=1$, and $\widehat{HFK}(L_0,u(L_0)-1)\cong \mathbb{F}[0]\oplus \mathbb{F}^t[k-1]$ for some $t\geq 0$ if $k>1.$

Hence, we have proved the proposition.\end{proof}
\subsection{Case 4: Suppose $w=P$ is in Xu's form, $w$ starts with $a_1$ and ends with $a_3$, $L=cl(w).$}
\begin{proposition} Let $k=\frac{l(UT(P))}{3},$ we have:
\begin{table}[H]
\centering
\caption{Classification of $\widehat{HFK}(L,u(L)-1)$ in Case 4}
\setlength{\tabcolsep}{6pt}
\renewcommand{\arraystretch}{1.30}
\begin{tabular}{|p{0.10\linewidth}|p{0.76\linewidth}|}
\hline \textbf{Value of $k$}  &$\widehat{HFK}(L,u(L)-1)$ \\
\hline
 $k\not=2$ 
& $\mathbb{F}^{rank(\widehat{HFK}(L_1,u(L_1)-1))+|L|-|L_1|-1}[-1]\oplus\mathbb{F}^{rank(\widehat{HFK}(L_2,u(L_2)-1))+|L|-|L_2|-1}[k-2],$ \newline$(L_1=cl(a_2w),$ $L_2=cl(a_2^{-1}w)).$ \\ \hline
 $k=2$& If $w\sim a_1^{n_1}w',$ $a_1w'$ nondecreasing positive, $n_1>3:$
 \newline
 $\mathbb{F}^{rank(\widehat{HFK}(L_1,u(L_1)-1))+|L|-|L_1|}[-1]\oplus\mathbb{F}^{rank(\widehat{HFK}_1(L_2,u(L_2)-1))+|L|-|L_2|-1}[0]$,
 \newline$(L_1=cl(a_1^{n_1-1}a_3a_1w'),$ $L_2=cl(a_1^{n_1-1}a_3^{-1}a_1w')).$
 \newline If $w$ contains no subword of the form $a_i^3:$
 \newline$\mathbb{F}^{|L|}[-1]\oplus\mathbb{F}^t[0]$ for some $t\geq 0.$
 \\ \hline
\end{tabular}
\label{tab:3braids-second-topP}
\end{table}
\end{proposition}
\begin{proof}
Let $L_+=cl(a_2P)$, $L_0=L$, $L_-=cl(a_2^{-1}P).$ We have $u(L_+)=u(L_-)=\overline{u}(L_0)$.

Take $u=u(L_+)$. By (7c) and (8), $\widehat{HFK}(L_+,u-1)\cong\mathbb{F}[k-2]\oplus \mathbb{F}^{t}[-1]$ for some $t\geq 0,$ and $\widehat{HFK}(L_-,u-1)$ is supported at Maslov gradings $0$ and $k-1$. Consider the exact triangle

\[
\begin{tikzcd}[row sep=large, column sep=large]
& H(L_{0}, u-1) \arrow[dr, "{[+0]}"] & \\
\widehat{HFK}(L_{-}, u-1) \arrow[ur, "{[-1]}"] \arrow[rr, "{[+0]}" above, leftarrow] & & \widehat{HFK}{(L_+, u-1)}
\end{tikzcd}
\]

From the exact triangle we see that $H(L_0,u-1)$ is supported at Maslov grading $k-2$ and $-1.$ In fact, when $k\not=2$, the exact triangle shows $$H_{m}(L_0,u-1)\cong\widehat{HFK}_{m}(L_+,u-1)\oplus\widehat{HFK}_{m+1}(L_-,u-1).$$

If $k=2$, suppose $w=a_1^{n_1}w'.$ 

If $n_1>2,$ let $L'_0=L,$ $L'_+=cl(a_1^{n_1-1}a_3a_1w')$, $L'_-=cl(a_1^{n_1-1}a_3^{-1}a_1w')$.

By (1), we have $a_1^{n_1-1}a_3a_1w'=\alpha^2a_3^{n_1-3}w',$ and by (2) we have $a_1^{n_1-1}a_3^{-1}a_1w'=a_1^{n_1-2}a_2^{-1}a_1^2w'\sim a_2^{-1}a_1^2 w'a_1^{n_1-2}.$

By analyzing word length, we see that $u(L'_+)=u(L'_-)=\overline{u}(L'_0)$. Set $u'=u(L'_+)$. By (6), we know that $\widehat{HFK}(L'_+,u'-1)$ is supported at Maslov grading $-1$ and by (8), $\widehat{HFK}(L'_-,u'-1)\cong\mathbb{F}[0]\oplus \mathbb{F}^t[1]$ for some $t\geq 0$. 

The exact triangle
\[
\begin{tikzcd}[row sep=large, column sep=large]
& H(L'_{0}, u'-1) \arrow[dr, "{[+0]}"] & \\
\widehat{HFK}(L'_{-}, u'-1)\arrow[ur, "{[-1]}"] \arrow[rr, "{[+0]}" above, leftarrow] & & \widehat{HFK}{(L'_+, u'-1)}
\end{tikzcd}
\]
shows that $H_m(L'_0,u'-1)\cong\widehat{HFK}_{m+1}(L'_-,u'-1)\oplus \widehat{HFK}_{m}(L'_+,u'-1)$.

Let $L(n_1,m_1,l_1,n_2,m_2,l_2)=cl(a_1^{n_1}a_2^{m_1}a_3^{l_1}a_1^{n_2}a_2^{m_2}a_3^{l_2})$. Let $S=\{L(n_1,m_1,l_1,n_2,m_2,l_2):\text{ each $n_i,m_i,l_i\leq 2$}\}.$

If $n_1\leq 2$ but some other $n_i,m_i,l_i>2,$ by conjugation it becomes the case when $n_1>2$. Thus, we may now assume each $n_i,m_i,l_i\leq 2.$

Assume $w=a_1^2a_2^{m_1}a_3^{l_1}a_1^{n_2}a_2^{m_2}a_3^{l_2}=a_1^2w'=a_1^2a_2w''a_3$. Let $L'_+=L,$ $L'_0=cl(a_1w')$, $L'_-=cl(w').$

We have $u(L'_-)=u(L'_+)-1=\overline{u}(L'_0)-1.$ Since $w'\sim \alpha w'',$ $L'_-$ is strongly quasi-positive and fibered, we have $\widehat{HFK}_m(L'_+,u(L'_+)-1)\cong H_m(L'_0,\overline{u}(L'_0)-1)$ for $m\not=-1$ by Lemma 4.2

From the skein triple $(L_+,L_-,L_0)$  we already known that the second-to-top terms of $L'_0,L'_+$ are supported at Maslov grading $0,-1$. At $m=-1,$ we have the exact sequence $$0\to \mathbb{F}\to H_{-1}(L'_0,\overline{u}(L'_0)-1)\to \widehat{HFK}_{-1}(L'_+,u(L'_+)-1)\to 0.$$ Then, $$rank(\widehat{HFK}_{-1}(L'_+,u(L'_+)-1))=rank(H_{-1}(L'_0,\overline{u}(L'_0)-1))-1=rank(\widehat{HFK}_{-1}(L'_0,u(L'_0)-1))-1$$ if $L_0$ has more components, and $$rank(\widehat{HFK}_{-1}(L'_+,u(L'_+)-1))=rank(H_{-1}(L'_0,\overline{u}(L'_0)-1))-1=rank(\widehat{HFK}_{-1}(L'_0,u(L'_0)-1))+1$$ if $L_0$ has fewer components.

Also, by conjugation, if any $n_i,m_i,l_i$ is 2, we may assume $n_1=2$. Hence, $rank(\widehat{HFK}_{-1}(L^*,u(L^*)-1))-|L^*|$ is constant for any $L^*\in S$. \textit{KnotInfo} shows that $L(2,1,1,1,1,2)=12n_{830}$ has second-to-top term $\mathbb
F[-1]\oplus\mathbb{F}^3[0]$. This means if $L\in S$, then $\widehat{HFK}(L,u(L)-1)\cong \mathbb
{F}^{|L|}[-1]\oplus \mathbb
F^t[0]$ for some $t\geq 0.$
\end{proof}
\subsection{Case 5: $w=a_1^{n}$, $n\geq 0$. $L=cl(w).$}
    
In this case, $cl(w)$ splits into a copy of $T(n,2)$ and a copy of the unknot $U$. For $n>1,$ $T(n,2)$ has top term $\mathbb{F}[0]$ and its second-to-top term is $\mathbb{F}^2[-1]$ for even $n$ and $\mathbb{F}[-1]$ or odd $n.$ Also, $\widehat{HFK}(U)\cong \mathbb{F}[0,0]$. Using the formula of $\widehat{HFK}$ for disjoint union [10], we have $\widehat{HFK}(L)\cong \widehat{HFK}(T(n,2))\otimes \widehat{HFK}(U)\otimes V$, where $V\cong [0,0]\oplus [-1,0]$. Thus, 
$$\widehat{HFK}(L,u(L)-1)\cong\begin{cases}
    0&\quad \text{if $n=0$ or 1,}\\
    \mathbb{F}[-1]\oplus\mathbb{F}[-2]&\quad \text{if $n>1$ and $n$ is odd,}
    \\\mathbb{F}^2[-1]\oplus\mathbb{F}^2[-2]&\quad \text{if $n>1$ and $n$ is even.}
\end{cases}$$
\section{Proof of Main Theorem and the Rank Conjecture}
Combining the formula $\sum_{m,s}(-1)^m \dim\widehat{HFK}_m(L,s)t^s\doteq (t^{\frac{-1}{2}}-t^{\frac{1}{2}})^{|L|-1}\Delta_{cl(w)}(t)$ and the propositions in Section 4, we can now fully determined $\widehat{HFK}(L,u(L)-1)$ for a closed 3-braid $L.$ This finishes the proof of Theorem 1.1.$\qed$

Corollary 1.3 also follows from the results in Section 4.
\begin{proof}[Proof of Corollary 1.3]
    \begin{lemma}
        If $w=\alpha P$ is a shortest word in Xu's form, $k=\lfloor\frac{l(UT(P))}{3}\rfloor>0,$ then $rank(\widehat{HFK}_{-1}(L,u(L)-1))\geq |L|.$
    \end{lemma}
    \begin{proof}
    If $P=P_1a_j^2P_2$ for some $j\in \{1,2,3\},$ with $P_1a_j,a_jP_2$ nondecreasing positive words, consider the skein triple $L_+=L, L_0=cl(\alpha P_1a_jP_2), L_-=cl(\alpha P_1P_2).$ 

As in Section 4.2, we have $u(L_-)=u(L_+)-1=\overline{u}(L_0)-1$, and $L_-$ is strongly quasi-positive and fibered. Thus, we have the exact sequence $0\to\mathbb{F}\to H_{-1}(L_0,\overline{u}(L_0)-1)\to \widehat{HFK}_{-1}(L_+,u(L_+)-1)\to 0$ by Lemma 4.2. Then, $rank(\widehat{HFK}_{-1}(L_+,u(L_+)-1))-|L_+|=rank(H_{-1}(L_0,\overline{u}(L_0)-1))-1-|L_+|=rank(\widehat{HFK}_{-1}(L_0,u(L_0)-1))-|L_0|.$ 

Hence, we reduce to the case $P=UT(P).$

If $l(UT(P))\equiv 1$ mod 3, we know that $\widehat{HFK}(L,u(L)-1)$ is supported at Maslov grading -1, and by [2], we have $rank(\widehat{HFK}(L,u(L)-1))\geq |L|.$ Hence, $rank(\widehat{HFK}_{-1}(L,u(L)-1))=rank(\widehat{HFK}(L,u(L)-1))\geq |L|.$

If $l(UT(P))\equiv 2$ mod 3, let $P=P'a_1a_2$. Let $L_+=L,$ $L_0=cl(\alpha P'a_1),$ $L_-=cl(\alpha P'a_1a_2^{-1})=cl(a_1^2P').$ 

We have $u(L_-)=u(L_+)-1=\overline{u}(L_0)-1.$ By Propositions 3.4 and 4.3, $\widehat{HFK}(L_-,u(L_-))$ is $\mathbb{F}[0]\oplus\mathbb{F}[k-1],$ $H(L_0,u(L_0)-1)\cong \mathbb{F}^t[-1]$, and $\widehat{HFK}(L_+,u(L_+)-1)\cong \mathbb{F}[k-1]\oplus \mathbb{F}^{t'}[-1]$, where $t,t'\geq 0.$

Now, take $u=u(L_+)$ and consider the exact sequences $$\cdots \to H_{m}(L_0,u-1)\to\widehat{HFK}_{m}(L_+,u-1)\xrightarrow{F_m}\widehat{HFK}_{m}(L_-,u-1)\xrightarrow{G_m}H_{m-1}(L_0,u-1)\to\cdots$$

If $k=1,$ let $L'_+=cl(\alpha a_1^2a_2a_3a_1a_2),$ $L'_0=L_+,$ $L'_-=cl(\alpha a_2a_3a_1a_2)$. From \textit{KnotInfo} [5], we know that $\widehat{HFK}(L'_+,u(L'_+)-1)\cong\mathbb{F}[0]\oplus\mathbb{F}[-1]$ and the skein exact triangle shows that $\widehat{HFK}_{-1}(L_+,u-1)\cong\mathbb{F}^2[-1]$.

If $k>1,$ we have the exact sequence $$0\to \mathbb{F}\to H_{-1}(L_0,u-1)\to \widehat{HFK}_{-1}(L_+,u-1)\to 0$$ because $G_0$ is injective as in Section 4.2. Also, $rank(\widehat{HFK}_{-1}(L_0,u(L_0)-1))\geq |L_0|.$ Thus, $$t'=t-1=rank(\widehat{HFK}_{-1}(L_0,u(L_0)-1))+|L_+|-|L_0|\geq |L_+|.$$ 

If $l(UT(P))\equiv 0$ mod 3, let $P=a_1a_2a_3 P'=a_1a_2a_3 P''a_3.$ Let $L_+=L,$ $L_0=cl(a_2 a_1a_2a_3P')$, $L_-=cl(a_2a_2a_3P').$

We have $u(L_-)=u(L_+)-1=\overline{u}(L_0)-1.$ As in Section 4.2, $L_-$ is fibered and strongly quasi-positive. By Lemma 4.2 we have an exact sequence $$0\to \mathbb{F}\to H_{-1}(L_0,\overline{u}(L_0)-1)\to \widehat{HFK}_{-1}(L_+,u(L_+)-1)\to 0.$$ Since $rank(\widehat{HFK}_{-1}(L_0,u(L_0)-1))\geq |L_0|,$ we have $$rank(\widehat{HFK}_{-1}(L_+,u(L_+)-1))=rank(H_{-1}(L_0,\overline{u}(L_0)-1))-1=rank(\widehat{HFK}_{-1}(L_0,u(L_0)-1))+|L_+|-|L_0|\geq |L_+|.$$\end{proof}
    \begin{lemma}
        If $w=a_2^{-1}P$ is a shortest word in Xu's form, $P$ starts with $a_1,$ $k=\lfloor\frac{l(UT(P))}{3}\rfloor>0,$ and $L=cl(w),$ then we have $rank(\widehat{HFK}_{k-1}(L,u(L)-1))\geq |L|.$ If $k=1,$ $rank(\widehat{HFK}_{0}(L,u(L)-1))\geq |L|+1.$
    \end{lemma}
    \begin{proof}
        Consider the skein triple $L_0=L,$ $L_+=cl(P),$ $L_-=cl(a_2^{-2}P).$ We have $u(L_-)=\overline{u}(L_0)=u(L_+)+1.$ Since $P$ must end with $a_1$ or $a_3,$ $P\sim P'$ where $P'$ is nondecreasing positive, $l(UT(P'))\equiv 0$ mod 3.
        
         If $k\geq 2$, by Propositions 3.4, 3.5, 4.9, 4.10, $\widehat{HFK}(L_+,u(L_+))\cong \mathbb{F}[k-1]\oplus \mathbb{F}[0]$, $H(L_0,\overline{u}(L_0)-1)\cong \mathbb{F}^t[k-1]\oplus \mathbb{F}[0]$, and $\widehat{HFK}(L_-,u(L_-)-1)\cong \mathbb{F}^{t'}[k]$, where $t,t'\geq 0.$

        We have the exact sequence $$0\to \widehat{HFK}_k(L_-,u(L_-)-1)\to H_{k-1}(L_0,\overline{u}(L_0)-1)\to \mathbb{F}\to 0.$$ We see that $$t'=t-1=rank(\widehat{HFK}_{k-1}(L_0,u(L_0)-1))-|L_0|+|L_-|.$$ Moreover, by [2], $t'\geq |L_-|,$ so $rank(\widehat{HFK}_{k-1}(L_0,u(L_0)-1))\geq |L_0|$ and hence $rank(\widehat{HFK}(L_0,u(L_0)-1))\geq |L_0|+1.$

        If $k=1,$ we see that $\widehat{HFK}(L_+,u(L_+))\cong \mathbb{F}^2[0]$, $H(L_0,\overline{u}(L_0)-1)\cong \mathbb{F}^t[0]$, and $\widehat{HFK}(L_-,u(L_-)-1)\cong \mathbb{F}^{t'}[1]$, where $t,t'\geq 0.$

        We have the exact sequence $$0\to \widehat{HFK}_1(L_-,u(L_-)-1)\to H_{0}(L_0,\overline{u}(L_0)-1)\to \mathbb{F}^2\to 0.$$ We see that $$t'=t-2=rank(\widehat{HFK}_0(L_0,u(L_0)-1))-1-|L_0|+|L_-|.$$ Moreover, by [2], $t'\geq |L_-|,$ so $rank(\widehat{HFK}_0(L_0,u(L_0)-1))\geq |L_0|+1$ and hence $rank(\widehat{HFK}(L_0,u(L_0)-1))\geq |L_0|+1.$
    \end{proof}
By [2], if $L$ is fibered then its second-to-top term has rank $\geq |L|.$ Thus, when $L$ has property $(*)$, the corollary is immediate. If $L$ is fibered and does not have property $(*),$ then either $L=cl(\alpha P),$ with $l(UT(P))\not \equiv 1$ mod 3 or $L=cl(a_2^{-1}P),$ with $k=\lfloor\frac{l(UT(P))}{3}\rfloor>1$; by Lemmas 5.1 and 5.2, the corollary holds in the two cases.

Thus, it suffices to consider the case when $L=cl(w),$ $w=P$ and $w$ starts with $a_1$ and ends with $a_3$ or $w=a_1^n.$ If $w_1=a_1^n$ the corollary is immediate. If $w=P$, and $w$ starts with $a_1$ and ends with $a_3$, let $k=\frac{l(UT(P))}{3}.$

If $k>2,$ let $L_+=cl(a_2P)$ and $L_-=cl(a_2^{-1}P).$ By Lemmas 5.1 and 5.2 and Proposition 4.11, we have $$rank(\widehat{HFK}_{k-2}(L,u(L)-1))=rank(\widehat{HFK}(L_-,u(L_-)-1))+|L|-|L_-|-1\geq |L|$$
 and $$rank(\widehat{HFK}_{-1}(L,u(L)-1))=rank(\widehat{HFK}(L_+,u(L_+)-1))+|L|-|L_+|-1\geq |L|.$$ Also, $\widehat{HFK}_m(L,u(L)-1)\cong \widehat{HFK}_{m+1}(L,u(L))\cong 0$ for $m\not=-1,k-2$, and $\widehat{HFK}(L,u(L))\cong \mathbb{F}[0]\oplus \mathbb{F}[k-1]$. Hence, the corollary holds.

 If $k=1,$ then by Lemmas 5.1 and 5.2 and Proposition 4.11,
 \begin{align*}
     rank(\widehat{HFK}_{-1}(L,u(L)-1))&=rank(\widehat{HFK}(L_-,u(L_-)-1))+rank(\widehat{HFK}(L_+,u(L_+)-1))+2|L|-2|L_-|-2
     \\&\geq 2|L|-1.
 \end{align*}
 If $rank(\widehat{HFK}_{-1}(L,u(L)-1))\geq 2,$ then the corollary holds. We could have $rank(\widehat{HFK}_{-1}(L,u(L)-1))<2$ only if $|L|=1.$ By Proposition 3.4, $\widehat{HFK}(L,u(L))$ is supported in a single Maslov grading. Thus, when $|L|=1,$ the corollary follows from [9]. 

If $k=2,$ suppose $L=L(n_1,m_1,l_1,n_2,m_2,l_2)$ and $w=a_1^{n_1}w'.$ 

If $n_1>2,$ then, as in Section 4.5, consider the skein triple $L'_0=L,$ $L'_+=cl(\alpha^2a_3^{n_1-3}w'),$ $L'_-=cl(a_2^{-1}a_1^2 w'a_1^{n_1-2}).$

Now, let $P'= w'a_1^{n_1-2}=a_2w''a_1^{n_1-2}$. consider $L''_+=L'_-,$ $L''_0=cl(a_2^{-1}a_1 P'),$ $L''_-=cl(w''a_1^{n_1-2})$. We see that $u(L''_-)<u(L''_+)-1=\overline{u}(L''_0)-1.$ Thus, by Lemma 4.1, \begin{align*}rank(\widehat{HFK}_1(L''_+,u(L''_+)-1))&=rank(H_1(L''_0,\overline{u}(L''_0)-1))\\&=rank(\widehat{HFK}_1(L''_0,u(L''_0)-1))+1+|L''_+|-|L''_0|\\&\geq |L''_+|+1\end{align*} because Lemma 5.2 has shown that $rank(\widehat{HFK}_1(L''_0,u(L''_0)-1))\geq |L''_0|.$

That is, by Proposition 4.11,
$$rank(\widehat{HFK}_{0}(L,u(L)-1))=rank(\widehat{HFK}_1(L_-,u(L_-)-1))+|L|-|L_-|-1\geq |L|$$
 and $$rank(\widehat{HFK}_{-1}(L,u(L)-1))=rank(\widehat{HFK}(L_+,u(L_+)-1))+|L|-|L_+|\geq |L|.$$
Also, $\widehat{HFK}_m(L,u(L)-1)\cong \widehat{HFK}_{m+1}(L,u(L))\cong 0$ for $m\not=-1,0$, and $\widehat{HFK}(L,u(L))\cong \mathbb{F}[0]\oplus \mathbb{F}[1]$. Hence, the corollary holds.

If each $n_i,m_i,l_i\leq 2$, then the discussion in Section 4.5 and the fact that $L(2,1,1,1,1,2)$ has second-to-top term $\mathbb{F}[-1]\oplus \mathbb{F}^3[0]$ reveals that $rank(\widehat{HFK}_0(L,u(L)-1))\geq rank(\widehat{HFK}_{-1}(L,u(L)-1))=|L|.$
\end{proof}

\noindent\section{Acknowledgment}

I would like to thank my advisor Professor Yi Ni for his guidance, and I would like to thank Professor Zhechi Cheng for addressing some of my questions about his work.

\end{document}